\documentclass[11pt]{article}
\usepackage{srcltx}

\usepackage{amsmath,amsfonts,latexsym}
\usepackage{latexsym,amsfonts,euscript,amssymb}
\numberwithin{equation}{section}

\newtheorem{theorem}{Theorem}[section]

\newtheorem{lemma}[theorem]{Lemma}

\newcommand{\cy}[1]{\mbox{$\langle #1 \rangle $}}

\newenvironment{proof}{\underline{Proof:}}{ \hfill $\Box$ 
                      \vspace{\baselineskip}}

\begin{document}

\pagestyle{myheadings}

\title{On the number of cyclic subgroups in finite groups}

\author{Wei Zhou}

\date{}
\maketitle

\begin{center}
	School of Mathematics and Statistics,\\
	Southwest University, Chongqing 400715, P. R. CHINA\\
	zh\_great@swu.edu.cn
	\end{center}

\footnotetext{Support by
	National Natural Science Foundation of China (Grant No. 11471266). 
	}

\renewcommand{\arraystretch}{0.1}

\begin{abstract}
	We study the number of cylic subgroups in finite groups and get that $G$ has $|G|-3$ cyclic subgroups if and only if $G \cong D_{10}$ or $Q_8$.
\end{abstract}

\baselineskip .65 true cm

\section{Introduction}
Let $G$ be a finite group  and $C(G)$ be the poset of cyclic subgroups of $G$. Sometimes $C(G)$ can decide the structure of $G$. For example, $G$ is an elementary abelian $2$-group if and only if $|C(G)|=|G|$. T\v{a}rn\v{a}uceanu \cite{Tar1, Tar2} classified the groups $G$ such that $|G|-|C(G)|=1$ or $ 2$.

In this note, we shall continue this study by describing the finite groups $G$ such that 
$$|C(G)|=|G|-3.$$
We prove that there are just two such groups: $D_{10}$ and $Q_8$.

For any finite group $G$, denote by $\pi_e(G)$ the set of all element orders of $G$, and donote by $\pi(G)$ the set of all prime divisors of $|G|$. For convenience, let $\pi_c(G)=\pi_e(G)-\pi(G)-\{1\}$. For any $i \in \pi_e(G)$, denote by $C_i(G)$ the set of all cylic subgroups of order $i$ in $G$, and denote $c_i(G)=|C_i(G)|$ ($c_i$ for short).  

\section{The main result}

\begin{lemma} \label{3p}
	Let $|G|=p_1^{a_1} \cdots p_r^{a_r}$ with $p_1< \cdots <p_r$. If $r\ge 3$, then $|G|-c(G)>p_r$.
\end{lemma}
\begin{proof}
For any finite group $G$, we know that
	\begin{equation*}
	\begin{split}
	&|G|=\sum_{k\in \pi_e(G)} c_k(\phi (k)), \\
	&|C(G)|=\sum_{k \in \pi_e(G)} c_k,
	\end{split}
	\end{equation*}
where $\phi$ is the Eucler function.
Hence 
\begin{equation}
|G|-|C(G)|=\sum_{k \in \pi_e(G)} c_k(\phi(k)-1).
\label{eq:1}
\end{equation}

	Let $G$ be a group such that $|G|-|C(G)|\le p_r$. 
By (\ref{eq:1}), we see that 
$$\sum_{k \in \pi_e(G)} c_k(\phi(k)-1)\le p_r.$$
By Cauch theorem, $c_{p_i} \ge 1$ for all $i \le r$. Hence we get that 
$$c_{p_1}(p_1-2)+c_{p_2}(p_2-2)+\cdots+c_{p_r}(p_r-2)+\sum_{s \in \pi_c(G)} c_s(\phi(s)-1)\le p_r.$$
Since $r \ge 3$, we get that $p_r \ge 5$. Thus $c_{p_r}=1$ and
$$c_{p_1}(p_1-2)+c_{p_2}(p_2-2)+\cdots+c_{p_{r-1}}(p_{r-1}-2)+\sum_{s \in \pi_c(G)} c_s(\phi(s)-1)\le 2.$$
So we get $r=3$, $p_2=3$ and $p_1=2$.
It follows that 
$$c_{3}+\sum_{s \in \pi_c(G)} c_s(\phi(s)-1)\le 2.$$

If $c_3=2$, then $\pi_c(G)=\emptyset$. Let $X_1, X_2$ be the two cylic subgroups of order $3$. Considering the action of $G$ on $\{X_1, X_2\}$, we see that $X_1$ is normalized by a Sylow $p_r$-subgroup, which implies  $3p_r\in \phi_c(G)$, a contradiction. Hence $c_3=1$. Similarly, we get that $3p_r \in \pi_c(G)$. But $c_{3p_r}(\phi(3p_r)-1)\ge 2(p_r-1)-1\ge 2(5-1)-1=7$, a contradiction.
\end{proof}

\begin{lemma} \label{2p}
	Let $|G|=p^aq^b$, where $p, q$ are primes such that $p<q$. Then $|G|-c(G)>q$ if $G \not \cong D_{2q}$, $C_6$, $D_{12}$, $C_6$ or $S_3$, and $|D_{2q}-|C(D_{2q})|=q-2$, $|C_6|-|C(C_6)|=2$, $|D_{12}|-|C(D_{12})|=2$. $|C_6|-|C(C_6)|=2$ and $|S_3|-|C(S_3)|=1$.
\end{lemma}
\begin{proof}
	Let $G$ be the group such that $|G|-c(G)\le q$. By (\ref{eq:1}), 
	\begin{equation}
	\label{eq:2}
c_p(p-2)+c_q(q-2)+\sum_{s \in \pi_c(G)} c_s(\phi(s)-1)\le q.
\end{equation}

	i) $q \ge 5$.
	Then $c_q=1$ and $$c_p(p-2)+\sum_{s \in \pi_c(G)} c_s(\phi(s)-1) \le 2.$$
	If 	$p\ne 2$, then $\pi_c(G)=\emptyset$, and $c_p(p-2) \le 2$.  It follows that $p=3$ and $c_3 \le 2$.  But we can find an element of order $3q$, a contradiction. So we get that $p=2$. We see that $\pi_e(G)=\{1,2, q\}$ or $\{1, 2, 2^2, q\}$. Thus $G$ has only one Sylow $q$-subgroup $Q$ isomorphic to $C_q$.

	If $\pi_e(G)=\{1,2, 4, q\}$, then $c_4 \le 2$. Thus $Q$ normalizes a cyclic subgroup of order $4$. This implies that $4q \in \pi_c(G)$, a contradiction. Hence $\pi_e(G)=\{1,2,q\}$. If $a\ge 2$, by considering the conjugate action of a Sylow $2$-group on $Q$, we can find an elmement of order $2q$, a contradiction. Hence $|G|=2q$, and $G=\cy{u, v| u^q=1, v^2=1, u^v=u^{-1}}\cong D_{2q}$.
	In this case, $|G|-|C(G)|=q-2$.

	ii) $q=3$ and $p=2$. Now (\ref{eq:2}) becomes that 
	$$c_3+\sum_{s\in \pi_c(G)} c_s(\phi(s)-1) \le 3.$$
	It follows that $c_3 \le 3$ and $\pi_e(G)\subseteq \{1, 2, 3, 4, 6\}$.
	
	If $\pi_e(G)=\{1,2,3,4, 6\}$, then $c_3=c_4=c_6=1$.
	And we get $12\in \pi_e(G)$, a contradiction. If $\pi_e(G)=\{1,2,3,4\}$, then $c_4 \le 2$. We can get that $12 \in \pi_e(G)$ if $c_4=1$. Hence $c_4=2$. Therefore, a Sylow $3$-group will normalizes a cyclic subgroup of order $4$, and we get that $12 \in \pi_e(G)$, a contradiction.

	If $\pi_e(G)=\{1,2,3, 6\}$, then $c_3+c_6 \le 3$. Then $c_3=1$ or $c_6=1$. Thus we can get a normal cyclic subgroup $X=\cy{x}$ of order $3$. Thus $|G:C_G(x)|\le 2$. 
	If $a\ge 3$, then $C_G(x)$ will contain a subgroup $L\cong C_2 \times C_2 \times C_3$. Since $c_6(L)=3$, we get a contradiction. Henc $a \le 2$. From $c_3 \le 2$, we get that $b=1$. Hence $|G| \le 12$, and $G \cong C_6$ or $D_{12}$.

	Now we need to consider the case that $\pi_e(G)=\{1,2,3\}$. Thus $c_3 \le 3$, and there are at most $6$ nontrivial $3$-element. It follows that a Sylow $3$-subgroup is isomorphic to $C_3$. By Sylow theorem, $c_3=1$. Thus the Sylow $3$-subgroup $Q$ is normal in $G$. Since $6 \not \in \pi_e(G)$, $a \le 1$, and $|G|=6$. Since $|C_6|-|C(C_6)|=2$ and $|S_3|-|C(S_3)|=1$, we get that $G \cong C_6$ or $S_3$ in this case.
	\end{proof}

	\begin{lemma} \label{1p}
		Let $|G|=2^a$. If $|G|-|C(G)|=2^{a}-3$, then $G \cong Q_8$.
	\end{lemma}
\begin{proof}
	Since $\phi(8)=4$, from (\ref{eq:1}), $exp(G)\le 4$. If $exp(G)=2$, then $c(G)=|G|$, a contradiction. Hence $exp(G)=4$, and $c_4=3$. We find a normal cyclic subgroup $X=\cy{x}$ of order $4$. Let $C=C_G(X)$. Then $|C|=2^a$ or $2^{a-1}$.
	
	We claim that $C/X$ is an elementary $2$-group. Otherwise, there exists an elemment $g\in C$ such that $gX$ is an  element of order $4$ in $C/X$. Hence $|g|=4$. Let $D=\cy{g,x}$. Then $D$ is abelian. Since $|gX|=4$, we get that $\cy{g}\cap X=1$, and $D\cong C_4 \times C_4$. But $c_4(C_4 \times C_4)>3$, a contradicton.
	
	Hence the Frattini subgroup $\Phi(C) \le X$, and $C$ is a $2$-group with cyclic Frattini subgroup. Suppose that $C$ is non-abelian. By \cite[Theorem 4.4]{Ber}, if $|\Phi(C)|>2$, there exists an element of order $2|\Phi(C)\ge 8$, a contradiction. Hence $\Phi(C)=C'\cong C_2$. 	By \cite[Lemma 4.2]{Ber}, $C=EZ(C)$ and $|E \cap Z(C)|=2$, where $E$ is an extra-special $2$-groups. By \cite[Theorem 3.13.8]{Hup}, $E=A_1*\cdots *A_m$, the central product of $A_i$, where $A_i\cong D_8$ or $Q_8$. Note that $c_4(Q_8)=3$ and $c_4(D_8)=1$ and $c_4(D_8*D_8)>3$. We get that $E=D_8$ for $X \not \le E$. Let $y \in E$ with $|y|=4$. Since $c_4=3$, we see that $Z(C) \cong C_4$ or $C_4 \times C_2$. If $Z(C)\cong C_4 \times C_2$, then $c_4(\cy{y, Z(C)})>3$, a contradiction. Hence $Z(C)=X\cong C_4$. But now $c_4(C)>3$, a contradiction. So we get that $C$ is abelian. 
	
	Since $c_4=3$, we get that $C \cong C_4$ or $C_4 \times C_2$. 
	Suppose that $C\cong C_4 \times C_2$. Since $c_4(C)=2$ and $c_4=3$, there exists $u \in G-C$ such that $|u|=2$. Let $C=\cy{x} \times \cy{w}$, where $|w|=2$. Thus $x^u=x^{\pm 1}$, and $w^u=w$ or $wx^2$, and we get $4$ groups. But none of the four groups satisfying $c_4=3$. Hence $C\cong C_4$, then $G \cong Q_8$.
\end{proof}

Combing Lemma \ref{3p}, \ref{2p} and \ref{1p}, we get our main result.	
	\begin{theorem}
		If $|C(G)|=|G|-3$, then $G \cong D_{10}$ or $Q_8$.
	\end{theorem}

\end{document}